\documentclass{conm-p-l}

\usepackage{amssymb,amsmath,graphicx,epsfig}

\newtheorem{theorem}{Theorem}[section]
\newtheorem{lemma}[theorem]{Lemma}
\newtheorem{corollary}[theorem]{Corollary}

\newtheorem{proposition}[theorem]{Proposition}

\theoremstyle{definition}
\newtheorem{definition}[theorem]{Definition}
\newtheorem{example}[theorem]{Example}

\theoremstyle{remark}
\newtheorem{remark}[theorem]{Remark}

\numberwithin{equation}{section}

\begin{document}

\newcommand{\cC}{{\mathcal{C}}}
\newcommand{\cO}{{\mathcal{O}}}
\newcommand{\cG}{{\mathcal{G}}}
\newcommand{\cB}{{\mathcal{B}}}
\newcommand{\cF}{{\mathcal{F}}}

\newcommand{\Q}{{\mathbb{Q}}}
\newcommand{\N}{{\mathbb{N}}}
\newcommand{\Z}{{\mathbb{Z}}}
\newcommand{\bP}{{\mathbb{P}}}

\title[Numerical semigroups and Groebner bases]{Characterization of gaps and elements of a numerical semigroup using Groebner bases}

\author{Guadalupe M\'arquez--Campos}
\address{Departamento de \'Algebra, Universidad de Sevilla. P.O. 1160. 41080 Sevilla, Spain.}
\email{gmarquez@us.es}
\author{Jos\'e M. Tornero}
\address{Departamento de \'Algebra, Universidad de Sevilla. P.O. 1160. 41080 Sevilla, Spain.}
\email{tornero@us.es}
\thanks{The authors were partially  supported by the grants FQM--218 and P08--FQM--03894, FSE and FEDER (EU)}

\date{October, 2013}

\begin{abstract}
This article is partly a survey and partly a research paper. It tackles the use of Groebner bases for addressing problems of numerical semigroups, which is a topic that has been around for some years, but it does it in a systematic way which enables us to prove some results and a hopefully interesting characterization of the elements of a semigroup in terms of Groebner bases.
\end{abstract}

\maketitle

\section{Numerical semigroups}

This paper deals with a very special family of semigroups. Recall that a semigroup is a pair $(X,\star)$, where $X$ is a set and $\star$ is an associative internal operation. Actually we will be considering monoids, that is, semigroups with unit element, but there are no substantial differences for our concerns. We will be particularly interested in the so--called numerical semigroups. Useful references for the basic concepts are \cite{RA,GS}.

\begin{definition}
A numerical semigroup is a semigroup $S \subset \Z_{\geq 0}$.
\end{definition}

\begin{example}
The first natural example of a numerical semigroup is the semigroup generated by a set $\{ a_1,...,a_k \} \subset \Z_{\geq 0}$, which is the set of linear combinations of these integers with non--negative integral coefficients:
$$
\langle a_1,...,a_k \rangle = \left\{ \lambda_1 a_1 + ... + \lambda_k a_k \; | \; \lambda_i \in \Z_{\geq 0} \right\}.
$$

It turns out that this example is in fact the general case for a numerical semigroup.
\end{example}

\begin{proposition}
Let $0 \leq a_1 \leq ... \leq a_k$ be integers such that $\gcd(a_1,...,a_k)=1$. Let us write $S = \langle \, a_1,...,a_k \, \rangle$. Then there exists $N \in \Z$ such that $x \in S$, for all $x \geq N$.
\end{proposition}

\begin{proof}
Let us write, from Bezout's Identity
$$
m_1a_1+...+m_ka_k=1,
$$
for some $m_i \in \Z$ and let
$$
P = \sum_{m_i \geq 0} m_ia_i > 0, \quad Q = \sum_{m_j \leq 0} m_ja_j \leq 0.
$$

We take an integer $t \geq (a_1-1)(-Q)$ and write it as $t=-Q(a_1-1)+k$, for a certain $k \geq 0$. We divide $k$ by $a_1$, 
$$
k = q a_1 + r, \mbox{ con $0 \leq r < a_1$}
$$
and then
\begin{eqnarray*}
t &=& -Q(a_1-1) + q a_1 + r \\
&=& -Q(a_1-1) + q a_1 + rP - rQ \\
&=& q \cdot a_1 + rP + (-Q)(a_1-1-r)
\end{eqnarray*}

This finishes the proof, as $a_1,P,-Q \in S$ and all their coefficients lie in $\Z_{\geq 0}$, therefore $t \in S$.
\end{proof}

\begin{remark}
If we had $\gcd(a_1,...,a_k)=d > 1$ the situation would be pretty analogous, taking into account that we should work in the ring $\Z d$ instead of $\Z$. This is why, in the sequel, when we talk about numerical semigroups we will assume that $\{a_1,...,a_k\}$ generate $\Z$ as an additive group.
\end{remark}

\begin{corollary}
Every numerical semigroup $S \neq \{ 0 \}$ can be written in the form $S = \langle a_1,...,a_k \rangle$.
\end{corollary}

\begin{proof}
Clearly if we take $a_1,...,a_k \in S$ then it must hold $\langle a_1,...,a_k \rangle \subset S$, so there is an $N \in \Z_{\geq 0}$ as in the proposition for $\langle a_1,...,a_k \rangle$. Then it is clear that $S$ is generated by
$$
\left\{ a_1,...,a_k \right\} \cup \left\{ x \in S \; | \; x < N \right\}.
$$
\end{proof}

As $S = \langle \, a_1,...,a_k \, \rangle$ is nothing but the set of non--negative integers that can be written as a linear combination (with non--negative coefficients) of $\{ a_1,...,a_k \}$, the elements of $S$ are often called {\em representable integers} (w.r.t. $\{ a_1,...,a_k \}$). In the same fashion the elements of the (finite) set $\Z_{\geq 0} \setminus S$ are called {\em non--representable integers}.

\begin{definition}
Some important invariants associated to a numerical semigroup $S$ are:

\begin{itemize}
\item The set of gaps, which is the finite set $\Z_{\geq 0} \setminus S$, noted $G(S)$.
\item The genus of $S$, noted $g(S)$, which is the cardinal of $G(S)$.
\item The Frobenius number of $S$ which is the maximum of $G(S)$, noted $f(S)$.
\item The set of sporadic elements, noted $N(S)$, which are elements of $S$ smaller than $f(S)$, that is $N(S) = S \cap [0,f(S)]$.
\item The cardinal of $N(S)$, noted $n(S)$ (this invariant has not a properly stablished name in the literature).
\item The multiplicity of $S$, noted $m(S)$, which is the smallest non--zero element in $S$ (obviously a generator in any case).
\item The dimension of $S$, noted $d(S)$, which is the smallest possible cardinal of a set of generators.
\item The conductor, noted $c(S)$, which is $f(S)+1 \in S$.
\end{itemize}
\end{definition}

\begin{remark}
The Frobenius number and its actual computation is a major problem in numerical semigroups. For semigroups of dimension $2$, $S = \langle a_1,a_2 \rangle$ it was solved by Sylvester \cite{Sylvester}, who proved
$$
f(S) = a_1a_2 - a_1 - a_2, \quad g(S) = \frac{c(S)}{2}.
$$

This problem, also known as {\em the money--changing problem} or {\em the nugget problem} has not an easy solution for $d(S) \geq 3$. Some closed formulas are known for certain cases, but Ram\'{\i}rez--Alfons\'{\i}n proved that the general problem is NP--hard under Turing reductions \cite{RA2}.
\end{remark}

\section{A characterization of elements and gaps in terms of Groebner bases}

\begin{remark}
The relationship between numerical semigroups and computational algebra tools can be traced back to the pioneering work of Herzog \cite{Herzog} and there is a great number of papers which build bridges between both subjects. This section is intended as a survey of a small subset of this rich relationship, containing the results we will be using afterwards in an organized and structured way.

Most results and related to Groebner bases can be found, for instance, in \cite{AL}, along with some results from this section, whose proofs we have included for the convenience of the reader.
\end{remark}

Let $b$ be a fixed natural number, $\{ a_1, a_2, a_3, ..., a_k \}$ a set of coprime non--negative integers, and  $\{ \sigma_1, \sigma_2, \sigma_3, ..., \sigma_k \}$ a set of variables taking values in $\Z_{\geq 0}$. We consider the equation:
$$
\sigma_1 a_1 + \sigma_2 a_2 + \sigma_3 a_3 + ... + \sigma_k a_k  = b.
$$

We introduce a new variable $x$ and rewrite the previous equation as:
$$
(x^{a_1})^{\sigma_1} (x^{a_2})^{\sigma_2} (x^{a_3})^{\sigma_3} ... (x^{a_k})^{\sigma_k} = x^b.
$$

Next we introduce new variables $y_j$, for $j=1,...,k$, and we set $x^{a_i} = y_i$, obtaining:
$$
y_{1}^{\sigma_1} y_{2}^{\sigma_2} y_{3}^{\sigma_3} ... y_{k}^{\sigma_k} = x^b
$$
where $\sigma_1, \sigma_2, \sigma_3, ..., \sigma_k$ are still unknown.

Consider the polynomial ideal
$$
I = \langle y_1 - x^{a_1},  y_2 - x^{a_2} ,  y_3 - x^{a_3} , ... ,  y_k - x^{a_k} \rangle \subset \Q[x, y_1, ..., y_k],
$$
and let $\cB = \{ g_1, g_2, g_3, ... , g_r\}$ a minimal Groebner basis of $I$ (not necessarily a reduced one), with respect to the usual lexicographic ordering $x > y_1 > y_2 > ...> y_k$.

Let us note $q_{i} = exp(g_i)$, the exponents of the polynomials $g_i$; and
$$
K_{q_{i}} = q_{i} + \Z_{\geq0} ^{k+1} \subset \Z^{k+1}_{\geq 0}.
$$

The main target is now to prove that there are one--to--one correspondences between
\begin{eqnarray*}
G(S) & \longleftrightarrow & \left[ \bigcap_{i} \overline{K_{q_{i}}} \right] \setminus \{ x =0 \} \subset \Z^{k+1}_{\geq0} \\
S & \longleftrightarrow & \left[ \bigcap_{i} \overline{K_{q_{i}}} \right] \cap \{ x =0 \} \subset \Z^{k+1}_{\geq0}
\end{eqnarray*}
in a very explicit way.

\vspace{.3cm}

In order to do that we will use two closely related maps:
\begin{eqnarray*}
\phi : \Q[y_1, y_2, ..., y_k] & \longrightarrow & \Q[x] \\
        y_j & \longmapsto & x^{a_j}
\end{eqnarray*}
and its extension
\begin{eqnarray*}
\widetilde{\phi} : \Q[x, y_1, y_2, ..., y_k] & \longrightarrow & \Q[x] \\
y_j & \longmapsto & x^{a_j}\\
x & \longmapsto & x
\end{eqnarray*}

\begin{lemma}
$\ker \left( \widetilde{\phi} \right) = I$.
\end{lemma}

\begin{proof} 
$I \subset \ker \left( \widetilde{\phi} \right)$ is clear. If we take $f(x,y_1,...,y_k) \in \ker \left( \widetilde{\phi} \right)$ we can perform Euclidean division w.r.t. $y_k,...,y_1$ to get an expression
$$
f = q_k(x,y_1,...,y_k)\left( y_k - x^{a_k} \right) + ... + q_1(x,y_1) \left( y_1 - x^{a_1} \right) + r(x)
$$
and $r(x)$ must lie in $\ker \left( \widetilde{\phi} \right)$, therefore $r(x)=0$.
\end{proof}

\begin{lemma}
$\cB$ is a binomial basis. Therefore the normal form of a monomial $x^N$, which we will write $N_\cB \left( x^N \right)$, is always a monomial.
\end{lemma}

\begin{proof}
It is well--known that the Groebner basis of a binomial ideal is again binomial \cite{Binomial}. Now assume we have a monomial $M_1$ and we want to reduce it w.r.t. a binomial $M_2 - M_3$, $M_2$ being the leading term.

If we cannot perform reduction, there is nothing to do. Otherwise $M_2 | M_1$ and then the remainder of the division is
$$
M_1 - \frac{M_1}{M_2} (M_2-M_3) = \frac{M_1M_3}{M_2},
$$
that is, a monomial.
\end{proof}

\begin{lemma}
Let $I$ be an ideal in a polynomial ring $R = k[x_1,...,x_n]$, $\cB$ a Groebner basis of $I$, and $g,f \in R$. Then $f \equiv g \mod I$ if and only if $N_\cB(f) = N_\cB(g)$.
\end{lemma}

\begin{proof}
It is a straightforward consequence of the fact that the mapping
$$
f \longmapsto N_\cB(f)
$$
is $k$--linear.
\end{proof}

\begin{theorem}
Let $I = \langle y_1 - x^{a_1},  y_2 - x^{a_2} ,  y_3 - x^{a_3} , ... ,  y_k - x^{a_k} \rangle \subset \Q[x, y_1, ..., y_k]$ and let $\cB$ be the reduced Groebner basis of I w.r.t. the lexicographic order $x < y_1 < ... < y_k$.

Then $f \in \Q[x]$ lies in $Im(\phi)$ if and only if there exists $h \in \Q[y_1, ..., y_k]$ such that $N_\cB(f) = h$. Should this be the case
$$
f = \phi (h) = h \left( x^{a_1},...,x^{a_k} \right).
$$
\end{theorem}

\begin{proof}
Assume $f = \phi(g) = g \left( x^{a_1},...,x^{a_k} \right)$. Then 
$$
f(x) - g(y_1,...,y_k) \in \ker \left( \widetilde{\phi} \right) = I
$$
and therefore
$$
N_\cB(f) = N_\cB(g) = h(x,y_1,...,y_k).
$$

Now, as $\cB$ does not depend on $x$, the elements of $\cB$ used in the computation of $N_\cB(g)$ must have their leading terms in $k[y_1,...,y_k]$. But, as we are using the lex ordering, in fact they must lie completely in $k[y_1,...,y_k]$. Therefore $h \in \Q[y_1,...,y_k]$.

\vspace{.3cm}

Assume now $N_\cB(f) = h \in \Q[y_1,...,y_k]$. Then $f-h \in I$ and therefore
$$
f(x) -h \left( y_1,...,y_k \right) = \sum_{i=1}^k g_i \left( x,y_1,...,y_k \right) \left( y_i - x^{a_i} \right),
$$
and doing $y_i = x^{a_i}$ we get $f = \phi(h) = h\left( x^{a_1}, ..., x^{a_k} \right)$.
\end{proof}

\begin{corollary}
If $x^N \in Im(\phi)$ then it is the image of a monomial $y_1^{\sigma_1}...y_k^{\sigma_k} \in \Q[y_1,...,y_k]$.
\end{corollary}

\begin{proof}
From the theorem $x^N \in Im(\phi)$ if and only if $N_\cB \left( x^N \right) = h$, with $x^N = \phi (h)$. As we saw previously, $h$ must be a monomial.
\end{proof}

\begin{remark}
Although we have chosen the lex ordering, one may note that in fact all we need for our argument is the fact that the ordering is an elimination one for the variable $x$. 

This idea will be most useful in the sequel, as it will allow us to change the ordering in order to meet our needs, and different orders will be used to tackle different problems.
\end{remark}

\begin{theorem}
Let $S = \langle a_1,...,a_k \rangle$, $I$ and $\cB$ as above, and let $N \in \Z_{\geq 0}$. Then
$$
N \in S \; \Longleftrightarrow \; x^N \in Im(\phi).
$$

Furthermore:
\begin{itemize}
\item If $N \in S$, then $N_\cB \left( x^N \right) = y_1^{\sigma_1}...y_k^{\sigma_k}$ and $N = \sigma_1 a_1 + ... + \sigma_k a_k$.
\item If $N \notin S$, then $N_\cB \left( x^N \right) = x^{\sigma_0} y_1^{\sigma_1}...y_k^{\sigma_k}$, with $\sigma_0 \neq 0$ and $N = \sigma_0 + \sigma_1 a_1 + ... +  \sigma_k a_k$.
\end{itemize}
\end{theorem}

\begin{proof}
Let $N \in S$. Then there are $\sigma_1,...,\sigma_k \in \Z_{\geq 0}$ with
$$
N = \sigma_1 a_1 + ... + \sigma_k a_k,
$$
and then 
$$
\begin{array}{lclcl}
x^N  & = & x^{a_1  \sigma_1 + a_2  \sigma_2+ ... + a_k \sigma_k} & = &(x^{a_1})^{\sigma_1} (x^{a_2})^{\sigma_2}  ... (x^{a_k})^{\sigma_k} \\
     & = & \phi(y_{1}^{\sigma_1})...\phi(y_{k}^{\sigma_k}) & = & \phi(y_{1}^{\sigma_1} ... y_{k}^{\sigma_k}), \\
\end{array}
$$
that is, $x^N \in Im(\phi)$.

\vspace{.3cm}

On the other hand, if $x^N \in Im(\phi)$, we know from the previous result
$$
x^N = \phi(h) = \phi \left( y_1^{\sigma_1}...y_k^{\sigma_k} \right) = 
\left( x_1^{a_1} \right)^{\sigma_1}...\left( x_k^{a_k} \right)^{\sigma_k},
$$
and $N = \sigma_1 a_1 + ... + \sigma_k a_k$. We already know as well that, in this case, $h = N_\cB \left( x^N \right)$.

Now, if $N \notin S$, we still know $N_\cB \left( x^N \right)$ is a monomial, say
$$
N_\cB \left( x^N \right) = x^{\sigma_0}y_{1}^{\sigma_1}...y_{k}^{\sigma_k}.
$$

As $N \notin S$, $N_\cB \left( x^N \right) \notin \Q[y_1,...,y_k]$, hence $\sigma_0 \neq 0$. As $N_\cB(f) - f \in I$ for all polynomials $f$,
$$
\exists h_i \in \Q[x,y_1, ..., y_k]  \; | \;  x^N - x^\sigma_0  y_1^{\sigma_1} ... y_k^{\sigma_k} = \sum_{i=1}^{k} h_i (y_i - x^{a_i}).
$$

We do then $y_i = x^{a_i}$ and
$$
x^N - x^{\sigma_0}  x^{a_1 \sigma_1} ... x^{a_k \sigma_k}=0
$$
hence $N=\sigma_0 + \sigma_1 a_1 +...+ \sigma_k a_k$.
\end{proof}

We are now ready to prove the one--to--one correspondences mentioned above.

\begin{theorem}
Let $S = \langle a_1,...,a_k \rangle \subset \Z_{\geq 0}$ be a numerical semigroup. Consider 
$$
I = \langle y_1 - x^{a_1},  y_2 - x^{a_2} ,  y_3 - x^{a_3} , ... ,  y_k - x^{a_k} \rangle \subset \Q[x, y_1, ..., y_k]
$$ 
and let $\cB =\{g_1,...,g_r\}$ be the reduced Groebner basis of I w.r.t. an elimination ordering for $x$, with $q_i = exp(g_i)$.

\begin{itemize}
\item The mapping
\begin{eqnarray*}
\mathcal{F}: G(S) &  \longrightarrow & \left[ \bigcap_{i} \overline{K_{q_{i}}} \right] \setminus \{ x =0 \} \subset \Z^{k+1}_{\geq0}\\
N & \longmapsto & exp \left(N_{\cB} \left(x^{N} \right) \right)
\end{eqnarray*}
is one--to--one.
\item The mapping
\begin{eqnarray*}
\mathcal{G} : S & \longrightarrow & \left[ \bigcap_{i} \overline{K_{q_i}} \right] \bigcap \{ x =0 \} \subset \Z^{k+1}_{\geq0} \\
M & \longmapsto &  exp \left( N_\cB \left(x^{M} \right) \right)
\end{eqnarray*}
is one--to--one.
\end{itemize}
\end{theorem}

\begin{proof}
Most of the results are more or less proved by now.

\vspace{.3cm}

\noindent {\bf I. ${\mathcal F}$ is surjective.} 

Let $(\sigma_{0}, \sigma_{1}, ..., \sigma_{k}) \in Im (\mathcal{F})$. Then there is some $N \in G(S)$ with 
$$
exp \left( N_{\cB} \left( x^N \right) \right) = (\sigma_{0}, \sigma_{1}, ..., \sigma_{k}).
$$

Being a normal form, it must hold
$$
(\sigma_{0}, \sigma_{1}, ..., \sigma_{k}) \in \bigcap_{i} \overline{K_{q_i}},
$$
and we previously saw $\sigma_0 \neq 0$.

On the other hand, take
$$
(\sigma_{0}, \sigma_{1}, ..., \sigma_{k}) \in \left[ \bigcap_{i} \overline{K_{q_i}} \right] \bigcap \{ x =0 \} = \left[ \overline{ \bigcup_{i} K_{q_i}} \right] \bigcap \{ x =0 \},
$$
so $(\sigma_{0}, \sigma_{1}, ..., \sigma_{k})$ does not lie in any $K_{q_i}$ and therefore
$$
x^{\sigma_0}  x^{a_1 \sigma_1} ... x^{a_k \sigma_k} = N_\cB \left( x^{\sigma_0}  x^{a_1 \sigma_1} ... x^{a_k \sigma_k} \right).
$$

Consider now $N = \sigma_0 + \sigma_1 a_1 + ... +  \sigma_k a_k$. Then
$$
\widetilde{\phi} \left( x^N \right) = \widetilde{\phi} \left( x^{\sigma_{0}}y_{1}^{\sigma_{1}} ... y_{k}^{\sigma_{k}} \right) \Longrightarrow x^N \equiv x^{\sigma_{0}}y_{1}^{\sigma_{1}} ... y_{k}^{\sigma_{k}}  \mod I.
$$

From a previous proposition
$$
N_\cB \left( x^N \right) = N_\cB \left( x^{\sigma_{0}}y_{1}^{\sigma_{1}} ... y_{k}^{\sigma_{k}} \right),
$$
and the fact that such $N$ is not in $S$ comes from the unicity of the normal form and the characterization of elements in $S$ in the previous theorem.

\vspace{.3cm}

\noindent {\bf II. ${\mathcal G}$ is surjective.} 

The proof goes parallel with the previous, with some necessary adjustments. Let us first consider $(\sigma_{0}, \sigma_{1}, ..., \sigma_{k}) \in Im (\mathcal{G})$. Then there is some $N \in S$ with 
$$
exp (N_{\cB} (x^N)) = (\sigma_{0}, \sigma_{1}, ..., \sigma_{k}).
$$

Being a normal form, it must hold
$$
(\sigma_{0}, \sigma_{1}, ..., \sigma_{k}) \in \bigcap_{i} \overline{K_{q_i}},
$$
and we have to see $\sigma_0 = 0$. But we get this from the previous theorem.

Let us see now 
$$
Im (\mathcal{G}) \supset \left[ \bigcap_{i} \overline{K}_{q_{i}} \right] \bigcap \{ x =0 \}, \; \forall i= 1,..,r.
$$

That is, for every $(0, \sigma_{1}, ..., \sigma_{k}) \in \bigcap_{i} \overline{K_{q_{i}}}$, we will find $M \in S$ with 
$$
exp \left(N_{\cB} \left(x^M \right)  \right) = (0, \sigma_{1}, ..., \sigma_{k}).
$$ 
But
$$
(0, \sigma_{1}, ..., \sigma_{k}) \in \bigcap \overline{K_{q_{i}}} \; \Longrightarrow \; y_{1}^{\sigma_{1}} ... y_{k}^{\sigma_{k}} = N_\cB \left(y_{1}^{\sigma_{1}} ... y_{k}^{\sigma_{k}} \right).
$$

We define $M = \sigma_1 a_1 + \sigma_2 a_2 + \sigma_3 a_3 + ... + \sigma_k a_k $ and from $\widetilde{\phi}$ we can see
$$
\widetilde{\phi} \left( x^M \right) =\widetilde{\phi} \left( y_{1}^{\sigma_{1}} ... y_{k}^{\sigma_{k}} \right) \Longrightarrow x^M \equiv y_{1}^{\sigma_{1}} ... y_{k}^{\sigma_{k}}  \mod I.
$$

This already implies $N_\cB \left( x^M \right) = N_\cB \left( y_{1}^{\sigma_{1}} ... y_{k}^{\sigma_{k}} \right)$.

\vspace{.3cm}

\noindent {\bf III. ${\mathcal F}$ and ${\mathcal G}$ are injective.} 

Should we have two non--negative integers $N_1, N_2$ with
$$
exp \left(N_{\cB} \left(x^{N_1} \right) \right) = exp \left(N_{\cB} \left(x^{N_2} \right) \right)
$$
this implies $x^{N_1} \equiv x^{N_2} \mod I$. Then there are polynomials $h_1,...,h_r$ with 
$$
x^{N_1} = x^{N_2} + \sum_{i=1}^{k} h_i (y_i - x^{a_i}),
$$
and doing $y_i = x^{a_i}$ we get $x^{N_1} = x^{N_2}$ and $N_{1} = N_{2}$.
\end{proof}

\begin{example}
Let us see a simple example, for a semigroup of dimension $2$, $S \langle5,7 \rangle$. Following Sylvester,
$$
f(S) = 5\cdot 7 - 5 - 7 = 23,
$$
and its set of gaps is
$$
G(S) = \{ 1,\ 2,\ 3,\  4,\ 6,\ 8,\ 9,\ 11,\ 13,\ 16,\ 18,\ 23\}.
$$

We consider then the ideal
$$
I= \langle y_1-x^5, \ y_2-x^7 \rangle \subset \Q[x,y_1,y_2],
$$
and we compute the (minimal) Groebner basis of $I$, using an elimination ordering for $x$. We have chosen the lex ordering $x>y_1>y_2$. The resulting Groebner basis is
$$
\cB =\left\{ -y_2^5+y_1^7,\ -y_1^3+y_2^2x,\ -y_2^3+y_1^4x,\ y_1x^2-y_2,\ y_2x^3-y_1^2,\ -y_1+x^5 \right\}.
$$

We can constuct now the sets 
$$
K_{q_i} = q_i + \Z_{\geq 0}^3 \subset \Z_{\geq 0}^3,
$$
with the exponents of the elements in $\cB$ (square points in the picture below):
$$
\begin{array}{lcllcllcl}
q_1 &=& (0,7,0), & q_2 &=& (1,0,2), & q_3 &=& (1,4,0), \\
q_4 &=& (2,1,0), & q_5 &=& (3,0,1), & q_6 &=& (5,0,0).
\end{array}
$$

\begin{center}
\includegraphics[width=90mm,height=60mm]{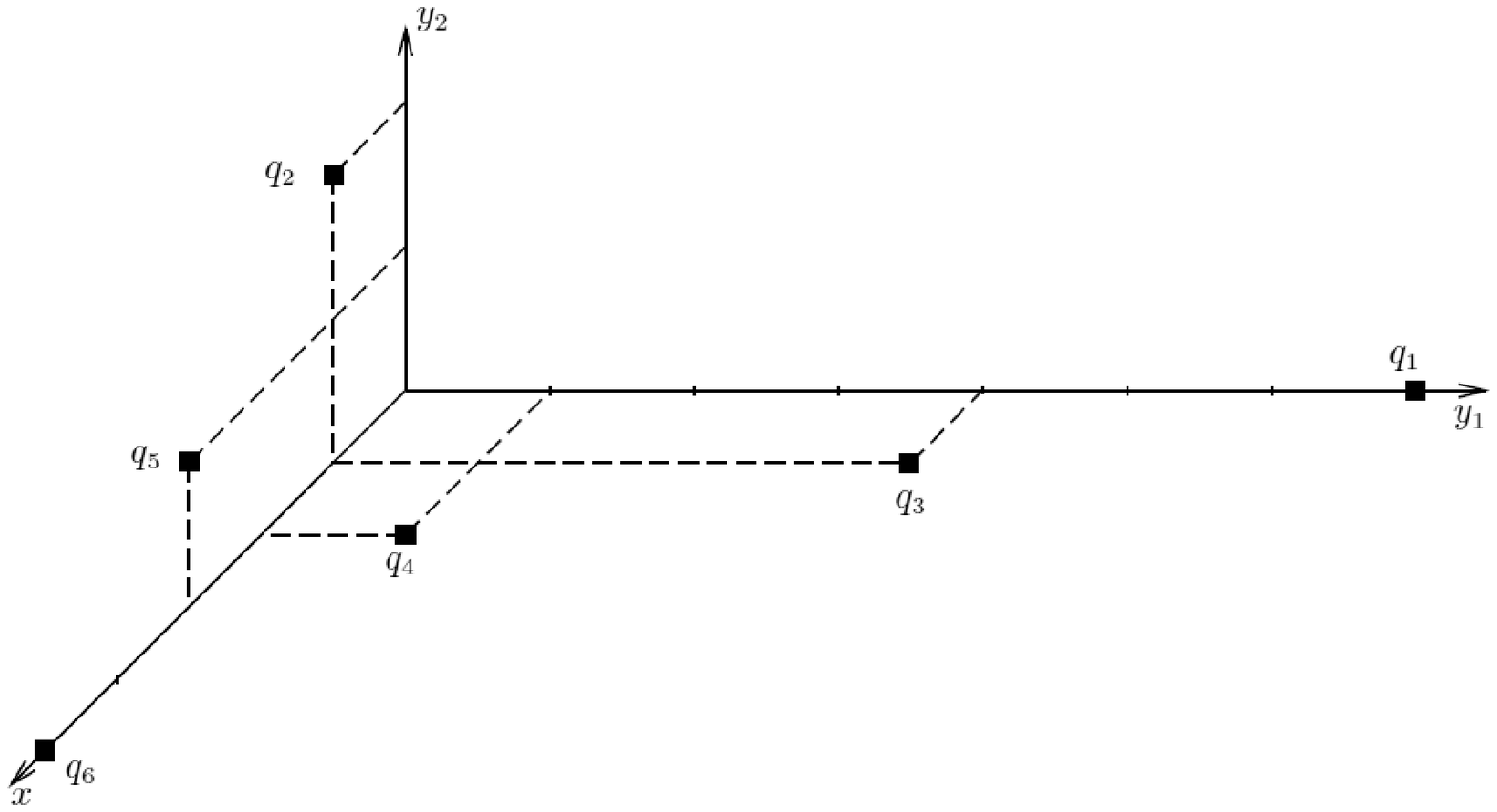}
\end{center}

\vspace{.3cm}

Now we check all elements from $G(S)$ and their one--to--one correspondence with
$$
\left[ \bigcap_{i} \overline{K_{q_{i}}} \right] \setminus \{ x =0 \} \bigcap \Z^3_{\geq0}
$$

In order to do this, we compute the normal form of all monomials $x^M$ with $M \in G(S)$, obtaining:

\vspace{.3cm}

\begin{tabular}{lclclclclcl}
$N_{\cB} (x^1)$ &=& $x$ &$\quad$& $N_{\cB} (x^2)$ &=& $x^2$ &$\quad$& $N_{\cB} (x^3)$ &=& $x^3$ \\
$N_{\cB} (x^4)$ &=& $x^4$ &$\quad$& $N_{\cB} (x^6)$ &=& $x y_1$ &$\quad$& $N_{\cB} (x^8)$ &=& $x y_2$ \\
$N_{\cB} (x^9)$ &=& $x^2 y_2$ &$\quad$& $N_{\cB} (x^{11})$ &=& $x y_1^2$ &$\quad$& $N_{\cB} (x^{13})$ &=& $x y_1 y_2$ \\
$N_{\cB} (x^{16})$ &=& $x y_1^3$ &$\quad$& $N_{\cB} (x^{18})$ &=& $x y_1^2 y_2$ &$\quad$& $N_{\cB} (x^{23})$ &=& $x y_1^3 y_2$
\end{tabular}

\vspace{.3cm}

These points can be seen in the lattice $\Z^3$, as expected (round points in the picture).

\begin{center}
\includegraphics[width=90mm,height=60mm]{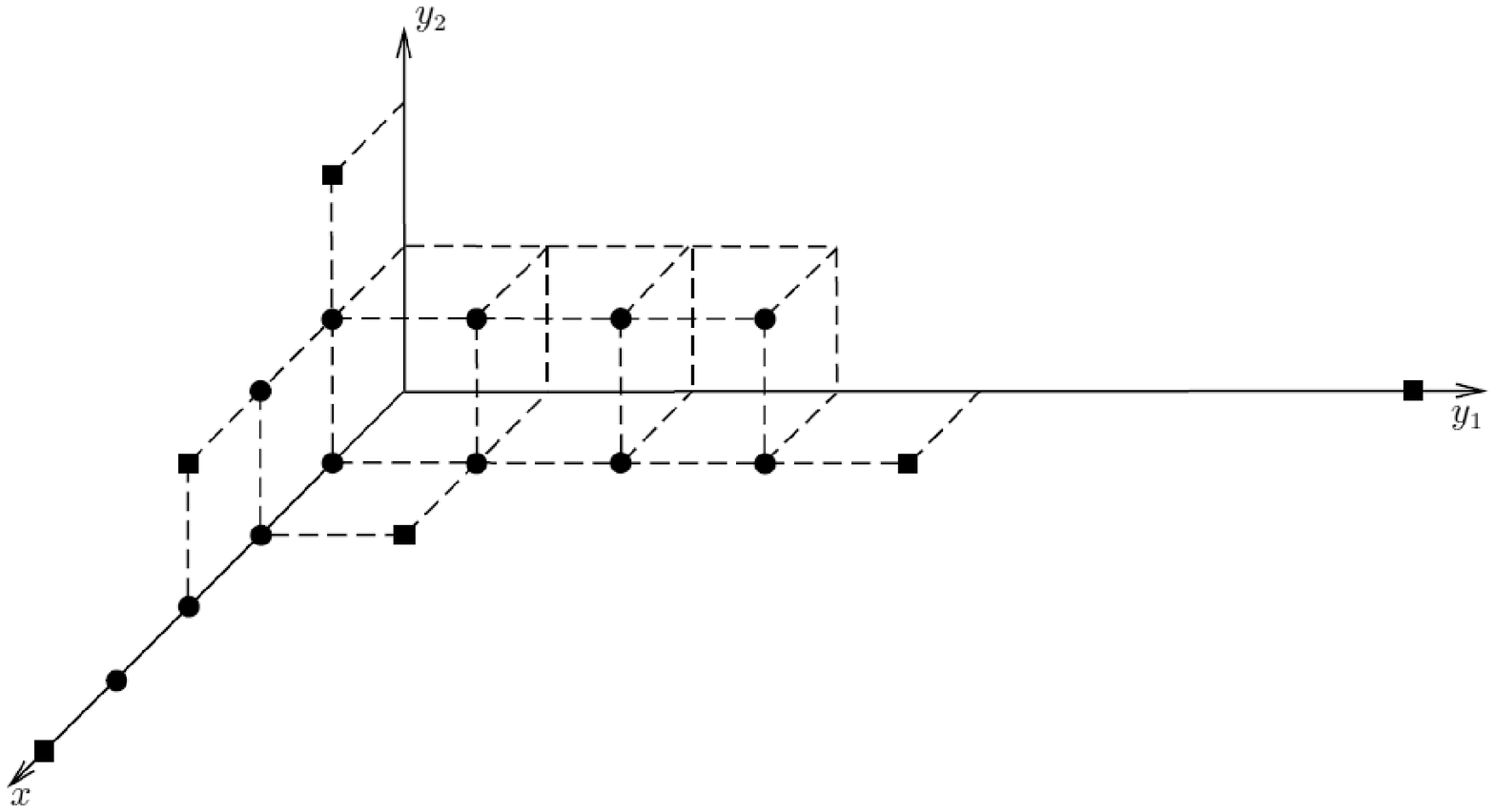}
\end{center}

\end{example}

\begin{example}
Let us consider now an example of dimension $3$. Let $S=  \langle 7,9,11 \rangle$. The Frobenius number of this numerical semigroup is:
$$
f(S) =26,
$$
and its set of gaps:
$$
G(S) = \{1,\ 2,\ 3,\ 4,\ 5,\ 6,\ 8,\ 10,\ 12,\ 13,\ 15,\ 17,\ 19,\ 24,\ 26\}.
$$

We can take the binomial ideal:
$$
I = \langle y_1 - x^{7},  y_2 - x^{9}, y_3 - x^{11}  \rangle \subset \Q[x, y_1, y_2, y_3]
$$
and find the Groebner basis $\cB$, using an elimination ordering w.r.t. $x$. For this example, we have taken the usual lexicographic ordering $x > y_1 > y_2 > y_3$. With this particular choice we get:
\begin{eqnarray*}
\mathcal{B} &=& \{y_{2}^{11}-y_{3}^9, -y_{2}^2+y_{3}y_1, y_{2}^9y_1-y_{3}^8, y_{2}^7 y_{1}^2-y_{3}^7, y_{2}^5 y_{1}^3-y_{3}^6, y_{2}^3 y_{1}^4-y_{3}^5, \\
    & &  y_{1}^5 y_{2}-y_{3}^4, -y_{2} y_{3}^3+y_{1}^6, -y_{2} y_{1}^2+y_{3}^2 x, -y_{1}^3+y_{3} y_{2} x, y_{2}^3 x-y_{1}^4,  \\
    & & y_{2}^2 y_{1}^2 x-y_{3}^3, -y_{3}^2+y_{1}^3 x, y_{2} x^2-y_{3}, y_{1} x^2-y_{2}, y_{3} x^3-y_{1}^2, -y_{1}+x^7\}
\end{eqnarray*}

We have to consider then, $q_{i} = exp(lt(g_i))$ where $g_i$ is the $i$--th polynomial in $\cB$, and take the corresponding set
$$
K_{q_{i}} = q_{i} + \Z_{\geq0} ^{k+1} \subset \Z^{k+1}_{\geq 0},
$$
in order to establish our bijections $\cF$ and $\cG$. In this case,
$$
\begin{array}{rclrclrcl}
q_1 &=& (0,0,11,0), \;\; & q_2 &=& (0,1,0,1), \; \; & q_3 &=& (0,1,9,0), \;\; \\
q_4 &=& (0,2,7,0), & q_5 &=& (0,3,5,0), & q_6 &=& (0,4,3,0), \\
q_7 &=& (0,5,1,0), & q_8 &=& (0,6,0,0), & q_9 &=& (1,0,0,2), \\
q_{10} &=& (1,0,1,1), & q_{11} &=& (1,0,3,0), & q_{12} &=& (1,2,2,0), \\
q_{13} &=& (1,3,0,0), & q_{14} &=& (2,0,1,0), & q_{15} &=& (2,1,0,0), \\
q_{16} &=& (3,0,0,1), & q_{17} &=& (7,0,0,0)
\end{array}
$$

Let us have a closer look to $\mathcal{F}$, so we are only interested in points of $\overline{\cup K_{q_i}}$ outside $x=0$. In order to represent the points, we will consider the subcases $x = \lambda$, with $\lambda \in \Z_{\geq 0}$. We have then:

\begin{itemize}
\item $x=1$. In this hyperplane we find several corners $q_i$, precisely
$$
q_9=(0,0,2), \; q_{10}=(0,1,1), \; q_{11}=(0,3,0), \; q_{12}= (2,2,0), \; q_{13}=(3,0,0)
$$

\noindent These points determine the elements of $\Z^4_{\geq 0} \setminus \cup K_{q_i}$, along with $(1,1,0,1) \in K_{q_2}$. As in the previous pictures, we will draw square points for points in $\cup K_{q_i}$, and round points for points outside $\cup K_{q_i}$, thus associated with a unique element of $G(S)$ by means of $\cF$:

\begin{center}
\includegraphics[scale=0.4]{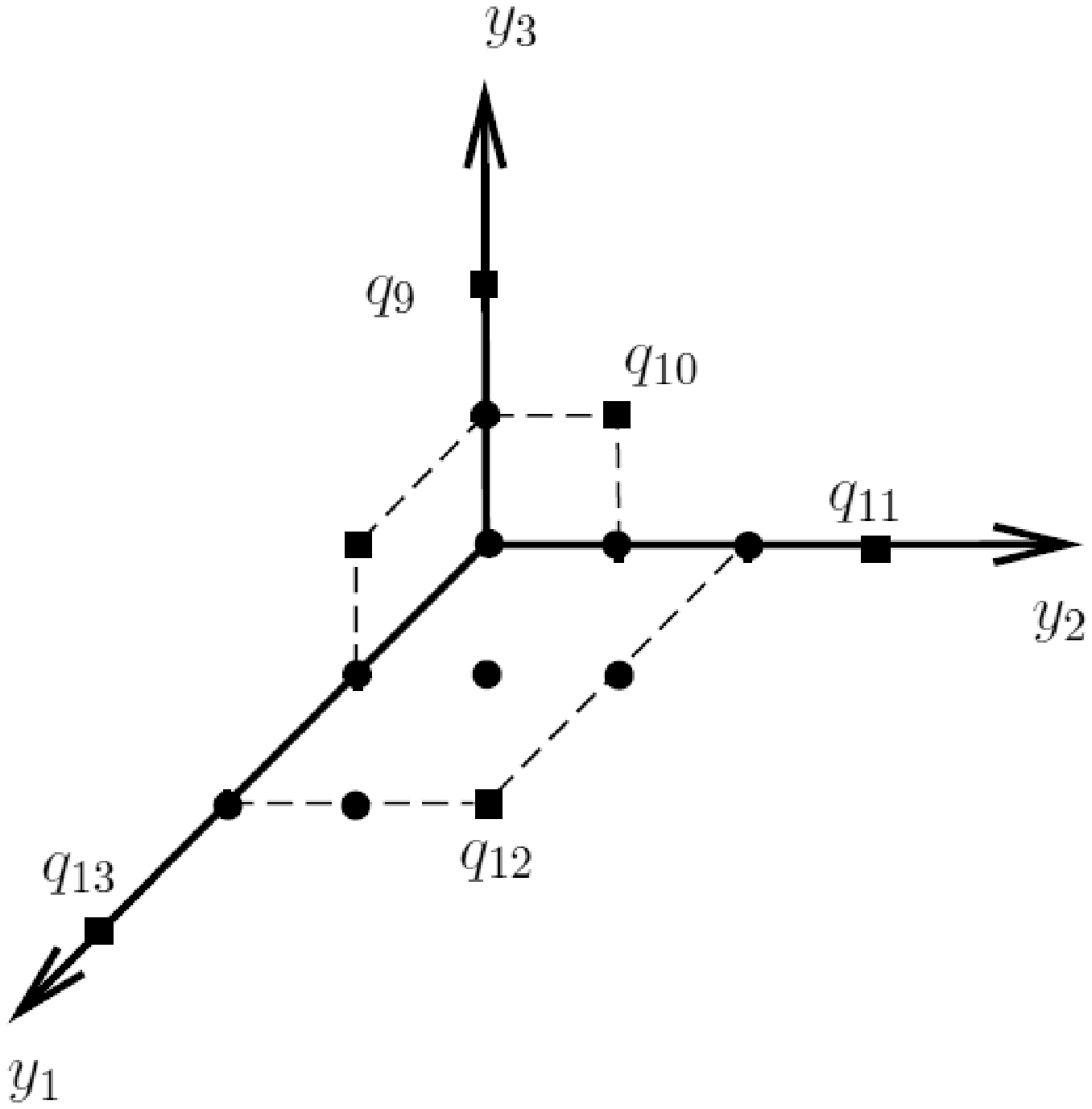}\\
\end{center}

\item At $x=2$ these are the points which determine the set:
$$
q_{14}=(0,1,0), \; \; q_{15}=(1,0,0), \; \; (0,0,2) \in K_{q_9}
$$

\begin{center}
\includegraphics[scale=0.4]{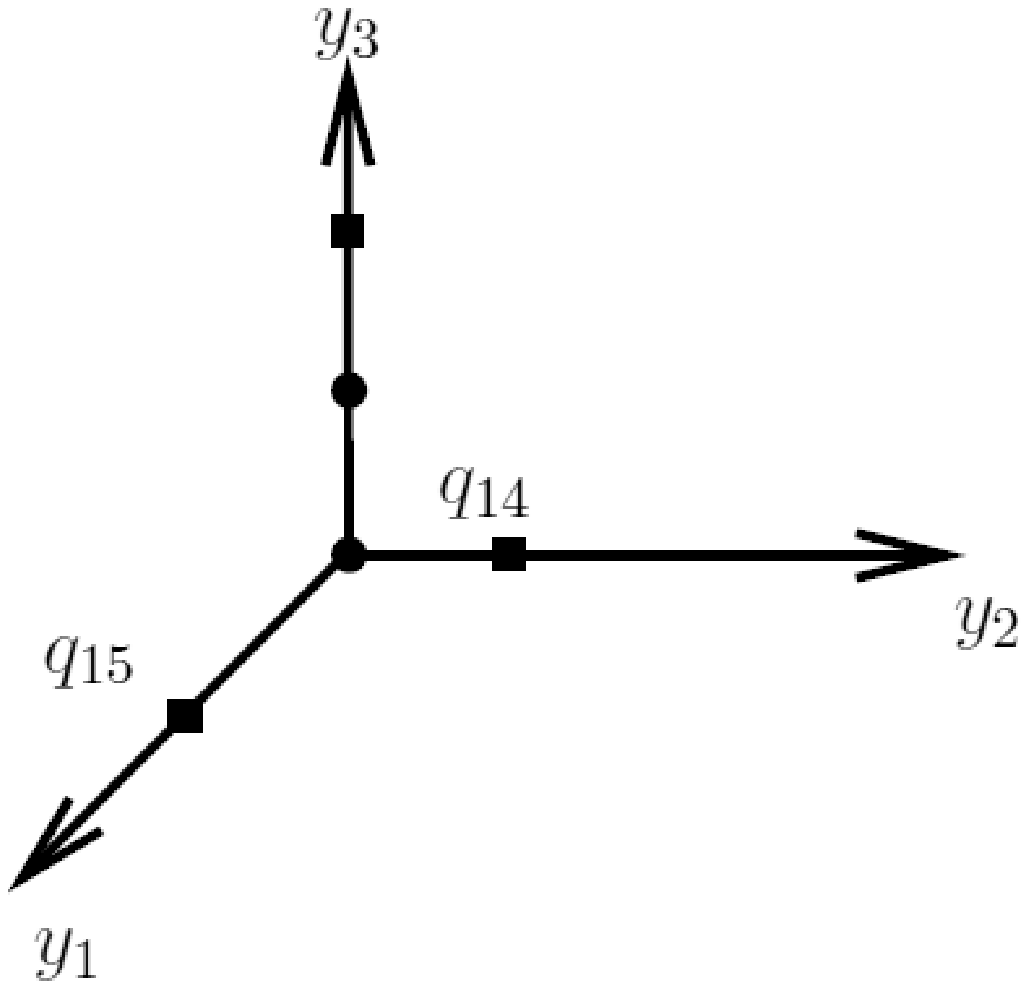}\\
\end{center}

\item At $x=3$, we have these points in $\cup K_{q_i}$
$$
q_{16}=(0,0,1), \; \; (1,0,0) \in K_{q_{15}}, \; \; (0,1,0) \in K_{q_{14}}
$$

\begin{center}
\includegraphics[scale=0.4]{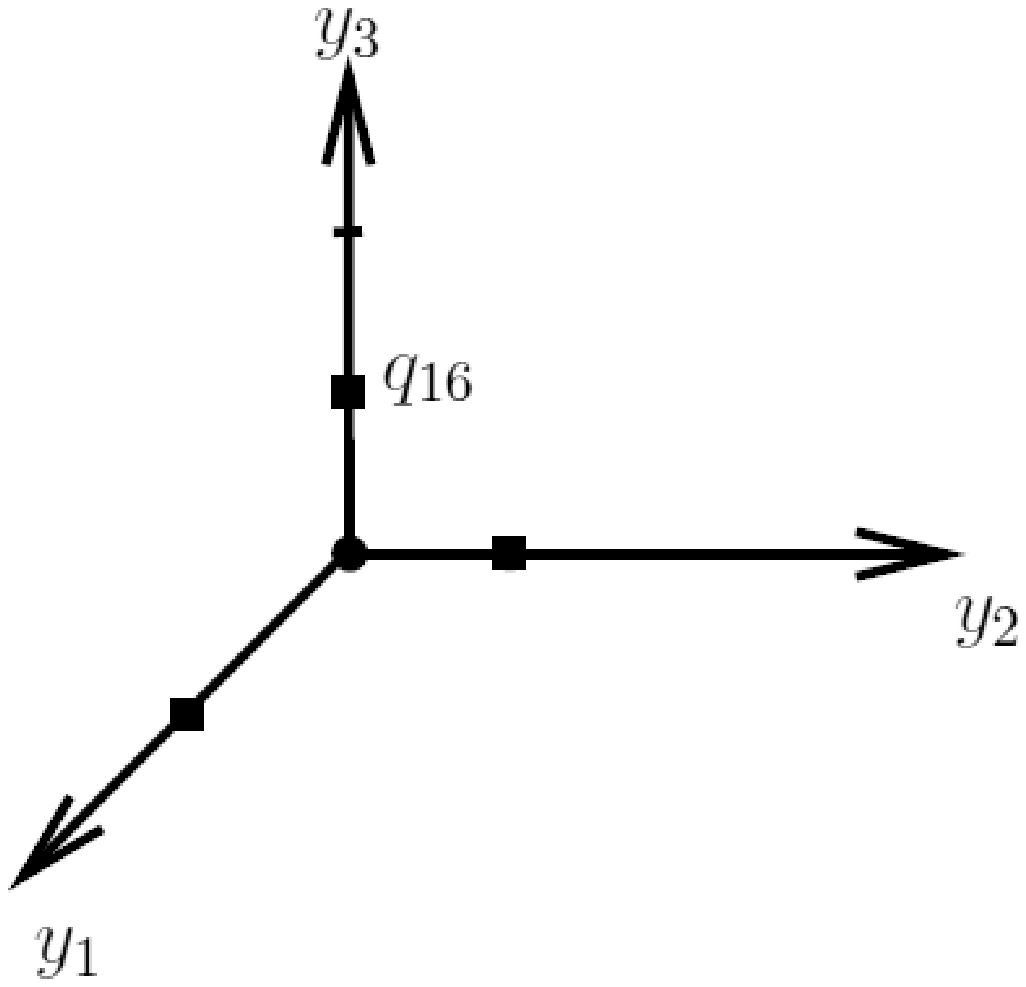}\\
\end{center}

\item At $x=4$, $x=5$ and $x=6$, the only relevant point is the origin, as $y_i < 1 $ for $i=1,2,3$ 

\item Last, in $x=7$ we have $(7,0,0,0) = q_{17}$, so this is, so to speak, the \emph{ceiling} for variable $x$.
\end{itemize}

If we compute the normal form of monomials $x^{n_i}$, where $n_i$ is the $i$--th gap, we get:
$$
\begin{array}{lcllcllcl}
N_{\mathcal{B}} (x^1) &=& x \quad & N_{\mathcal{B}} (x^2) &=& x^2  \quad & N_{\mathcal{B}} (x^3) &=& x^3 \\
N_{\mathcal{B}} (x^4) &=& x^4 & N_{\mathcal{B}} (x^5) &=& x^5 & N_{\mathcal{B}} (x^6) &=& x^6 \\
N_{\mathcal{B}} (x^8) &=& x y_1 & N_{\mathcal{B}} (x^{10}) &=& x y_2 & N_{\mathcal{B}} (x^{12}) &=& x y_3 \\
N_{\mathcal{B}} (x^{13}) &=& x^2 y_3 & N_{G} (x^{15}) &=& x y_1^2 & N_{\mathcal{B}} (x^{17}) &=& x y_1 y_2 \\
N_{\mathcal{B}} (x^{19}) &=& x y_2^2 \quad & N_{\mathcal{B}} (x^{24}) &=& x y_1^2 y_2 \quad & N_{\mathcal{B}} (x^{26}) &=& x y_1 y_2^2
\end{array}
$$
\end{example}

\begin{remark}
Therefore, for a given $N \geq 0$ we have a representation
$$
exp \left( N_\cB (x^N) \right) =(\sigma_0, ..., \sigma_k) \; \Longrightarrow \; N = \sigma_0 + \sum_{i=1}^{k} a_i \sigma_i,
$$
which is unique, provided 
$$
(\sigma_0, ..., \sigma_k) \in \left[ \bigcap_{i} \overline{K_{q_{i}}} \right],
$$
and which determines easily whether $N \in S$ or not, simply by looking at $\sigma_0$.

Let us consider $N \in S$. A very interesting function related to $S$ (actually to the set $\{a_1,...,a_k\}$) is the so--called denumerant, which is defined by
\begin{eqnarray*}
d: S & \longrightarrow & \Z \\
N & \longmapsto & d(N) = \sharp \left\{ (y_1,...,y_k) \in \Z^k_{\geq 0} \; | \; N = \sum_{i=1}^k y_i a_i \right\}
\end{eqnarray*}

That is, $d(N)$ is nothing but the number of different representations of $N$ as a non--negative integral linear combination of $\{a_1,...,a_k\}$. The notion of denumerant was ﬁrst introduced by Sylvester \cite{Sylvester2}.

On the other hand, if we take $N \in S$, aside from the representation mentioned above, we might have lots of others, only all of them in $\cup K_{q_i}$. Just in case someone is tempted, where is no relationship between $d(N)$ and
$$
\sharp \left\{ q_i \; | \;  x^N \in K_{q_i} \right\},
$$
as an easy example may show.

Take as before $S = \langle 5,\ 7 \rangle$, and consider $N=100$. The number of non--negative representations $100=5y_1+7y_2$ can be computed quickly, as all integral representations are given by
$$
y_1 = 7n + 6, \;\  y_2 = -5n + 10, \;\  n \in \Z.
$$

Hence only $n=0,1,2$ are suitable, and therefore $d(100)=3$. Analogously for $N=327$ we get
$$
y_1 = 7n + 1, \;\  y_2 = -5n + 46, \;\  n \in \Z.
$$
hence we get $d(N)=10$. However, both elements lie in the same quadrant $K_{q_6}$, and only in this one.

\end{remark}

\section{A first application: a bound ``\'a la Wilf''}

One of the most celebrated open problems in numerical semigroups is the so--called Wilf's Conjecture \cite{Wilf}, which states a very simple relationship among three important invariants:

\vspace{.3cm}

\noindent {\bf Wilf's Conjeture.--} Let $S$ be a numerical semigroup. Then
$$
c(S) \leq e(S)n(S).
$$

\vspace{.3cm}

That is to say, the conjecture fixes a lower bound for the proportion of sporadic elements among those non--negative integers smaller than the conductor of $S$: they must represent, at least, $1/e(S)$ of them.

The conjecture has been proved for a number of particular cases (see for instance \cite{Kaplan, Sammartano}). It has also been checked for semigroups of genus up to $50$ by M. Bras--Amor\'os \cite{Bras}.

What follows is our approximation to the problem of relating $n(S)$ and $c(S)$, using the techniques introduced above, resulting in a couple of bounds of different nature.

\vspace{.3cm}

\noindent {\bf Notation.--} Given rational positive numbers $\alpha_1,...,\alpha_n$, we define
\begin{eqnarray*}
P(\alpha_1,...,\alpha_n) &=& \left\{ (x_1,...,x_n) \in \mathbb{Z}_{>0}^n \; | \; \frac{x_1}{\alpha_1}+...+\frac{x_n}{\alpha_n} \leq 1 \right\} \\
Q(\alpha_1,...,\alpha_n) &=& \left\{ (x_1,...,x_n) \in \mathbb{Z}_{\geq 0}^n \; | \; \frac{x_1}{\alpha_1}+...+\frac{x_n}{\alpha_n} \leq 1 \right\} \\
\end{eqnarray*}
and
\begin{eqnarray*}
p(\alpha_1,...,\alpha_n) &=& \sharp \left( P(\alpha_1,...,\alpha_n) \right) \\
q(\alpha_1,...,\alpha_n) &=& \sharp \left( Q(\alpha_1,...,\alpha_n) \right) \\
\end{eqnarray*}

That is, $q(\alpha_1,...,\alpha_n)$ is the number of integral points in the tetrahedron limited by the coordinate hyperplanes and 
$$
\frac{x_1}{\alpha_1}+...+\frac{x_n}{\alpha_n} = 1,
$$
as $p(\alpha_1,...,\alpha_n)$ is the same thing, but discarding the points in the coordinate faces. 

\vspace{.3cm}

The relationship between these two quantities is given by the following result.

\begin{lemma}
Under the previous conditions, if we call
$$
\alpha = \frac{1}{\alpha_1} + ... + \frac{1}{\alpha_n},
$$
then
$$
q(\alpha_1,...,\alpha_n) = p(\alpha_1(1+\alpha),...,\alpha_n(1+\alpha)).
$$
\end{lemma}

\begin{proof}
Let us consider the following map:
\begin{eqnarray*}
\Phi: Q(\alpha_1,...,\alpha_n) & \longrightarrow & P(\alpha_1(1+\alpha),...,\alpha_n(1+\alpha)) \\
(x_1,...,x_n) & \longmapsto & (x_1+1,...,x_n+1)
\end{eqnarray*}

It is well--defined, as
$$
\sum_{i=1}^n \frac{x_i+1}{\alpha_i(1+\alpha)} = \frac{1}{1+\alpha} \left( \sum_{i=1}^n \frac{x_i}{\alpha_i} +
\sum_{i=1}^n \frac{1}{\alpha_i} \right) \leq 1,
$$
hence $Im(\Phi) \subset P(\alpha_1(1+\alpha),...,\alpha_n(1+\alpha))$. 

$\Phi$ is clearly injective, but is also surjective because
$$
\sum_{i=1}^n \frac{x_i}{\alpha_i(1+\alpha)} \leq 1 \; \Longleftrightarrow \; 
\sum_{i=1}^n \frac{x_i}{\alpha_i} \leq 1 + \alpha \; \Longleftrightarrow \; 
\sum_{i=1}^n \frac{x_i-1}{\alpha_i} \leq 1.
$$ 
\end{proof}

The hunt for a good, simple estimate of $q(\alpha_1,...,\alpha_n)$ and $p(\alpha_1,...,\alpha_n)$ led to several results \cite{LY1,LY2,LY3,WY,XY1,XY2,XY3}, finally put together in the GLY Conjeture, named after its authors Granville, Lin and Yau.

\vspace{.3cm}

\noindent {\bf GLY Conjecture.--} Assume $n \geq 3$ and let $\alpha_1 \geq ... \geq \alpha_n \geq 1$ be real numbers. Then:

\begin{itemize}
\item (Weak estimate) We have
$$
n! \cdot p(\alpha_1,...,\alpha_n) \leq (\alpha_1-1)...(\alpha_n-1),
$$
with equality if and only if $\alpha_n=1$.
\item (Strong estimate) Given $n$, there is a constant $C(n)$ such that, for $\alpha_n \geq C(n)$ we have
$$
n! \cdot p(\alpha_1,...,\alpha_n) \leq A^n_n + (-1) \frac{S_1^{n-1}}{n}A_{n-1}^n + \sum_{l=2}^{n-1} (-1)^l
\frac{S^{n-1}_l}{\left( \begin{array}{c} n-1 \\ l-1 \end{array} \right)} A^{n-1}_{n-l},
$$
where $S^{n-1}_l$ are the Stirling numbers, and $A^l_i$ are polynomials in $\alpha_1,...,\alpha_l$ with degree $i$.
\end{itemize}

\vspace{.3cm}

The weak version was finally proved by Yau and Zhang \cite{YauZhang}. In the same paper, the authors claim the strong version has been checked computationally up to $n\leq 10$. The fact is the conjecture might be checked for a particular $n$, but the state--of--the--art has not changed since. According to the authors, the case $n=10$ took weeks to be completed.

\vspace{.3cm}

Assume then we have a numerical semigroup $S = \langle a_1,...,a_k \rangle$ and let us consider the binomial ideal associated to $S$, as in the previous section
$$
I = \langle \, y_i - x^{a_i} \; | \; i=1,...,k \, \rangle \subset \Q[x,y_1,...,y_k].
$$

Let us fix an elimination ordering for $x$ and let us compute the Groebner basis $\cB$ and the corresponding sets $K_{q_i}$. As we know

$$
S \;  \stackrel{1:1}{\longleftrightarrow} \;  \left[ \bigcap_{i} \overline{K_{q_i}} \right] \bigcap \{ x =0 \} \subset \Z^{k+1}_{\geq0} 
$$

Therefore we may note

\begin{eqnarray*}
n(S) &=& \sharp \{ a \in S \; | \; a \leq f(S) \} \\
&=& \sharp \left\{ (0,y_1,...,y_k) \in \left[ \bigcap_{i} \overline{K_{q_i}} \right] \; | \; \sum y_i a_i \leq f(S) \right\} \subset \Z^{k+1}_{\geq0},
\end{eqnarray*}

\noindent which proves that $n(S)$ is less or equal to the number of integral points in the tetrahedron defined by the coordinate hyperplanes and 
$$
\frac{y_1}{f(S)/a_1} + ... + \frac{y_k}{f(S)/a_k} \leq 1.
$$

That is,
$$
n(S) \leq q \left(  \frac{f(S)}{a_1},...,\frac{f(S)}{a_k} \right),
$$
and from the previous lemma and the Weak estimate of the GLY Conjecture, 
\begin{eqnarray*}
n(S) &\leq& p \left(  \frac{f(S)}{a_1} \left( 1 + \sum \frac{a_i}{f(S)} \right),...,\frac{f(S)}{a_k} \left( 1 + \sum \frac{a_i}{f(S)} \right) \right) \\
&=& p \left(  \frac{f(S) + \sum a_i }{a_1},...,\frac{f(S) + \sum a_i}{a_k} \right) \\
&\leq& \frac{1}{k!} \prod_{j=1}^k \left( \frac{f(S) + \sum a_i}{a_j} - 1 \right) \\
&=& \frac{1}{k!\ a_1...a_k} \prod_{j=1}^k \left(f(S) + \sum_{i \neq j} a_i \right)
\end{eqnarray*}

We have then proved:

\begin{proposition}
Given a numerical semigrup $S = \langle a_1,...,a_k \rangle$, we have
$$
n(S) \leq \frac{1}{k!\ a_1...a_k} \prod_{j=1}^k \left( f(S) + \sum_{i \neq j} a_i \right)
$$
\end{proposition}

Hence we have actually proved a result which is, in certain sense, a reverse of Wilf's Conjecture, as we have actually proved an upper bound for $n(S)$ in terms of:

\begin{itemize}
\item $k$, which is an upper bound for $e(S)$, although it can be assumed from the beginning to be $e(S)$.
\item $f(S)$.
\item The generators of $S$.
\end{itemize}

\begin{remark}
Note that, if we make $k=2$ in the statement above, we get
$$
n(S) \leq \frac{1}{2a_1a_2} \left( a_1a_2-a_1 \right)\left( a_1a_2-a_2 \right) = \frac{(a_1-1)(a_2-1)}{2} = n(S),
$$
from Sylvester's result. So, in this case (where we cannot apply the GLY weak estimate, as it is valid for $k \geq 3$), the formula is still valid. Not only that, but the bound turns out to be an equality.
\end{remark}

\begin{remark}
Accuracy of the bound. In the following tables there are some examples of numerical semigroups, with the relevant information concerning the previous result.

\vspace{.3cm}

As it becomes plain, the bound gets less and less accurate as $n$ grows. A significant number of examples could be of help in order to look for a conjectural improvement, we are still far from that.

\vspace{.5cm}

\begin{center}
\begin{tabular}{|c|l|c|c|c|c|}
\hline
Dimension & Generators & $f(S)$ & $n(S)$ & Bound & Bound$/n(S)$ \\
\hline
\hline
$3$ & $\{5,6,11\}$ & $19$ & $8$ & $19$ & $\simeq 2.375$ \\
\hline
$3$ & $\{5,6,19\}$ & $14$ & $5$ & $10$ & $\simeq 2.000$ \\
\hline
$3$ & $\{5,7,16\}$ & $18$ & $8$ & $14$ & $\simeq 1.750$ \\
\hline
$3$ & $\{5,7,23\}$ & $18$ & $7$ & $13$ & $\simeq 1.857$ \\
\hline
$3$ & $\{6,9,20\}$ & $43$ & $21$ & $44$ & $\simeq 2.095$ \\
\hline
$3$ & $\{7,9,38\}$ & $40$ & $18$ & $28$ & $\simeq 1.555$ \\
\hline
$3$ & $\{7,9,40\}$ & $38$ & $16$ & $26$ & $\simeq 1.625$ \\
\hline
$3$ & $\{7,9,47\}$ & $40$ & $17$ & $28$ & $\simeq 1.647$ \\
\hline
$3$ & $\{7,48,50\}$ & $143$ & $62$ & $94$ & $\simeq 1.516$ \\
\hline
$3$ & $\{8,9,47\}$ & $46$ & $20$ & $31$ & $\simeq 1.550$ \\
\hline
$3$ & $\{8,9,55\}$ & $47$ & $20$ & $32$ & $\simeq 1.600$ \\
\hline
$3$ & $\{9,10,53\}$ & $61$ & $28$ & $42$ & $\simeq 1.500$ \\
\hline
\end{tabular}

\vspace{.5cm}

\begin{tabular}{|c|l|c|c|c|c|}
\hline
Dimension & Generators & $f(S)$ & $n(S)$ & Bound & Bound$/n(S)$ \\
\hline
\hline
$4$ & $\{7,11,34,37\}$ & $38$ & $14$ & $50$ & $\simeq 3.571$ \\
\hline
$4$ & $\{7,11,23,24\}$ & $27$ & $8$ & $31$ & $\simeq 3.875$ \\
\hline
$4$ & $\{7,11,23,17\}$ & $31$ & $11$ & $38$ & $\simeq 3.454$ \\
\hline
$4$ & $\{11,25,37,56\}$ & $101$ & $40$ & $110$ & $\simeq 2.750$ \\
\hline
$4$ & $\{11,25,37,115\}$ & $104$ & $42$ & $120$ & $\simeq 2.857$ \\
\hline
$4$ & $\{11,25,37,104\}$ & $101$ & $40$ & $111$ & $\simeq 2.775$ \\
\hline
$4$ & $\{9,13,19,21\}$ & $33$ & $10$ & $35$ & $\simeq 3.500$ \\
\hline
$4$ & $\{9,10,21,35\}$ & $43$ & $18$ & $59$ & $\simeq 3.277$ \\
\hline
$4$ & $\{8,11,13,15\}$ & $25$ & $8$ & $31$ & $\simeq 3.875$ \\
\hline
$4$ & $\{13,15,31,63\}$ & $81$ & $34$ & $94$ & $\simeq 2.764$ \\
\hline
$4$ & $\{13,16,33,56\}$ & $86$ & $34$ & $98$ & $\simeq 2.882$ \\
\hline
$4$ & $\{13,15,31,63\}$ & $81$ & $34$ & $94$ & $\simeq 2.764$ \\
\hline
\end{tabular}

\vspace{.5cm}

\begin{tabular}{|c|l|c|c|c|c|}
\hline
Dimension & Generators & $f(S)$ & $n(S)$ & Bound & Bound$/n(S)$ \\
\hline
\hline
$5$ & $\{7,11,31,34,37\}$ & $30$ & $9$ & $86$ & $\simeq 9.555$ \\
\hline
$5$ & $\{7,15,18,26,34\}$ & $38$ & $17$ & $112$ & $\simeq 6.588$ \\
\hline
$5$ & $\{9,10,21,35,43\}$ & $34$ & $11$ & $99$ & $\simeq 9.000$ \\
\hline
$5$ & $\{10,19,31,37,54\}$ & $65$ & $25$ & $154$ & $\simeq 6.160$ \\
\hline
$5$ & $\{8,11,13,15,20\}$ & $25$ & $11$ & $72$ & $\simeq 6.545$ \\
\hline
$5$ & $\{8,11,13,15,25\}$ & $20$ & $6$ & $53$ & $\simeq 8.833$ \\
\hline
$6$ & $\{10,19,31,37,54,65\}$ & $63$ & $24$ & $366$ & $\simeq 15.250$ \\
\hline
$6$ & $\{10,19,31,37,54,63\}$ & $65$ & $26$ & $382$ & $\simeq 14.692$ \\
\hline
\hline
\end{tabular}
\end{center}

\end{remark}

We will try a different approach, taking advantage of the catalogue of Groebner basis at our disposal. Let us take the lexicographic elimination ordering given by
$$
x < y_k < ... < y_2 < y_1.
$$

Let us fix an integer $\alpha \geq 0$, and consider
$$
n(S,\alpha) = \sharp \{ x \in S \; | \; x \leq \alpha \},
$$
so in particular $n(S,f(S))=n(S)$. We also have, as before
$$
n(S,\alpha) = \sharp \left\{ Y=(y_1,...,y_k) \in \Z_{\geq 0}^k \; | \; 
y_i \geq 0, \; \sum a_iy_i \leq \alpha, \; Y \notin \left[ \bigcup_{i} K_{q_i} \right] \right\}
$$

Let us call, without further mention of the bijection $\cG$, $N(S,\alpha)$ the previous set, whose number of points is $n(S,\alpha)$. Mind that 
$$
Y = \left( y_1,...,y_k \right) \in N(S,\alpha) \; \Longrightarrow \; 0 \leq y_1 \leq \frac{\alpha}{a_1}
$$

Assume first that we have $\alpha \geq a_1a_2$, the other case will be dealt with later and with some important differences. That is, for now we will consider
$$ 
\frac{\alpha}{a_1} - a_2 \geq 0.
$$

We are going to compute a bound for the set $N(S,\alpha)$ in two stages:

\begin{itemize}
\item First, we will construct a truncated prism $C$ over a $(k-1)$--hypercube, which will contain all points in $N(S,\alpha)$ with $0 \leq y_1 \leq \alpha/a_1 - a_2$.
\item After this, we will construct a pyramid $D$ which will contain the rest of the integral points in $N(S,\alpha)$, and we will compute with no great difficulty the number of integral points inside this pyramid.
\end{itemize}

Let us construct $C$. First note that the binomials $y_i^{a_1} - y_1^{a_i} \in I$, for all $i=2,...,k$. As their exponents are
$$
(0,...,0,\stackrel{(i)}{a_1},0,...,0) \in \Z_{\geq 0}^k,
$$
we have that 
$$
(0,...,0,\stackrel{(i)}{a_1},0,...,0) \in \left[ \bigcup_{i} K_{q_i} \right] \subset \Z_{\geq 0}^k.
$$
and then

\begin{eqnarray*}
N(S,\alpha) &=& \left\{ Y=(y_1,...,y_k) \in \Z_{\geq 0}^k \; | \; 
y_i \geq 0, \; \sum a_iy_i \leq \alpha, \; Y \notin \left[ \bigcup_{i} K_{q_i} \right] \right\} \\
& \subset & \left\{ Y=(y_1,...,y_k) \in \Z_{\geq 0}^k \; | \; 
0 \leq y_i < a_1, \mbox{ for } i=2,...,k \right\} = C_0,
\end{eqnarray*}
which is clearly a prism over a $(k-1)$--hypercube.

This bound could fit for all the set $N(S)$, but we will try to do better in the following way. First, we will compute at which point(s) the prism $C_0$ {\em hits the wall} defined by 
$$
a_1y_1 + ... + a_ky_k = \alpha.
$$

\begin{center}
\includegraphics[scale=0.3]{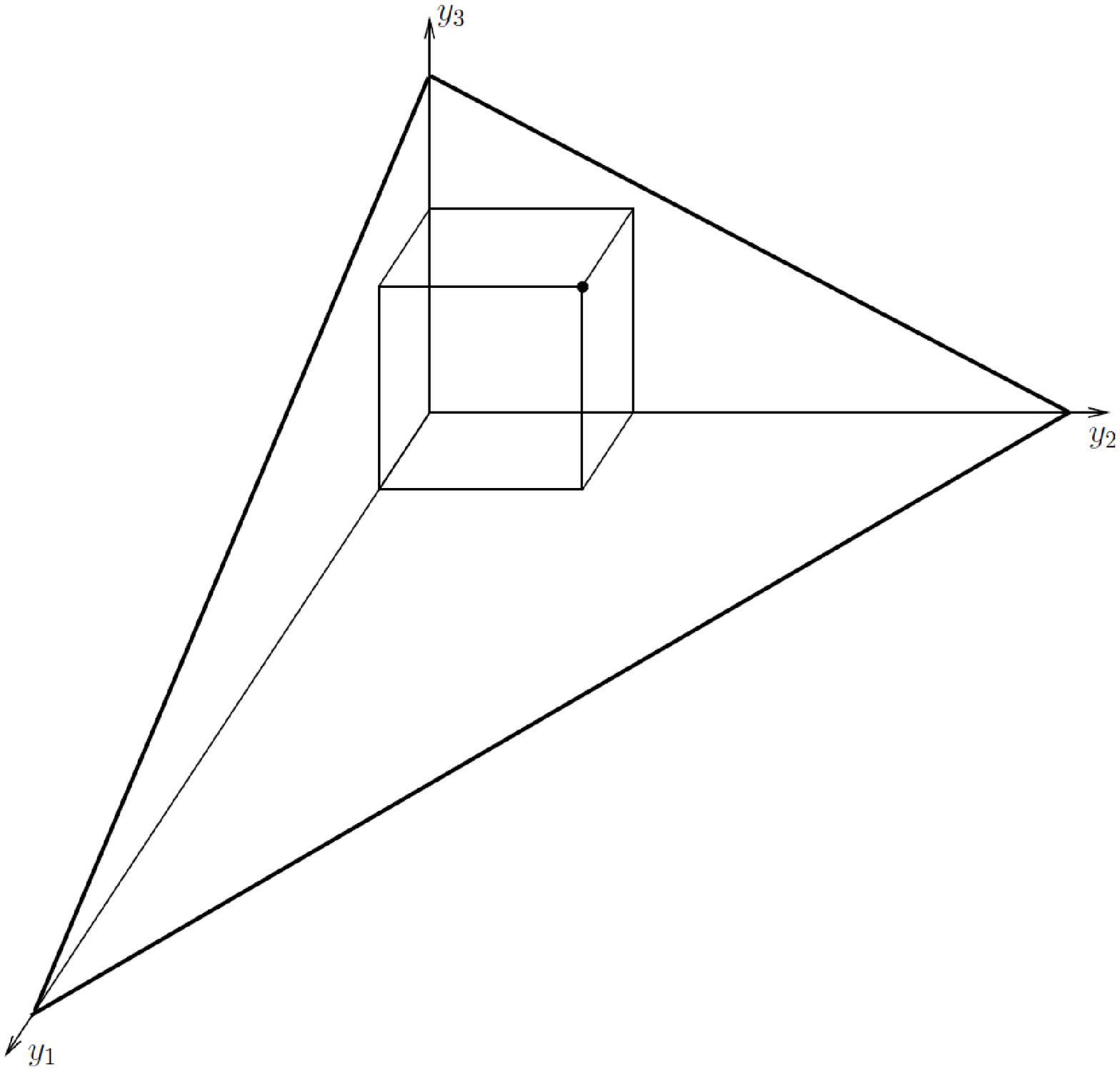}\\
\end{center}

If we set $y_2=...=y_k=a_1$, then the (integral) boundary of $C_0$ and the wall meet at the point
$$
R_0 = \left( \frac{\alpha}{a_1} - \sum_{i=2}^k a_i,a_1,...,a_1 \right).
$$

In order to construct a pyramid which is easier to work with, we will take a little more from $C_0$ before truncating it, so we will actually get out of $N(S,\alpha)$. More precisely, we will get to the point
$$
R_1 = \left( \frac{\alpha}{a_1} - a_2,a_1,...,a_1 \right).
$$

So, for now, what we have is
$$
N(S,\alpha) \bigcap \left\{ y_1 \leq \frac{\alpha}{a_1} - a_2 \right\}
$$
is contained in the truncated prism defined by
$$
C = \left\{ (y_1,...,y_k) \in \Z^k_{\geq 0} \; | \; y_1 \leq \frac{\alpha}{a_1} - a_2, \; y_i < a_1 \mbox{ for } i=2,...,k \right\}
$$

\begin{center}
\includegraphics[scale=0.3]{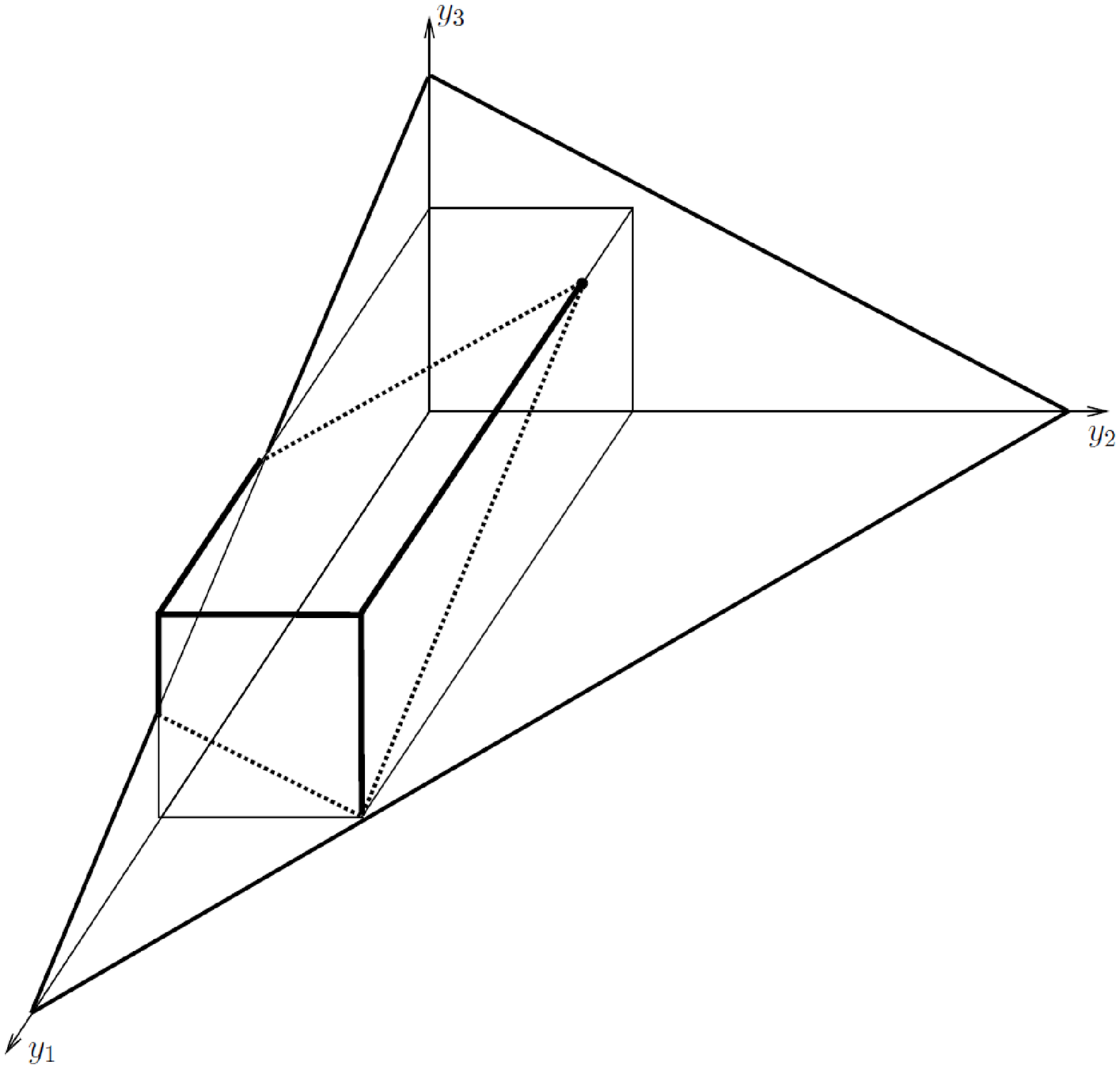}\\
\end{center}

Let us now build our pyramid $D$, which will have as its basis a $(k-1)$--convex on the hyperplane
$$
y_1 = \frac{\alpha}{a_1} - a_2,
$$
and its vertex at 
$$
V = \left( \frac{\alpha}{a_1},0,...,0 \right).
$$

The precise description is
$$
D = \left\{ V + \lambda_1 \left( -a_2,0,...,0 \right) + \sum_{i=2}^k \lambda_1\lambda_i (0,...,0,\stackrel{(i)}{a_1},0,...,0) \; | \; 0 \leq \lambda_i \leq 1, \; \forall i \right\}.
$$

\begin{center}
\includegraphics[scale=0.3]{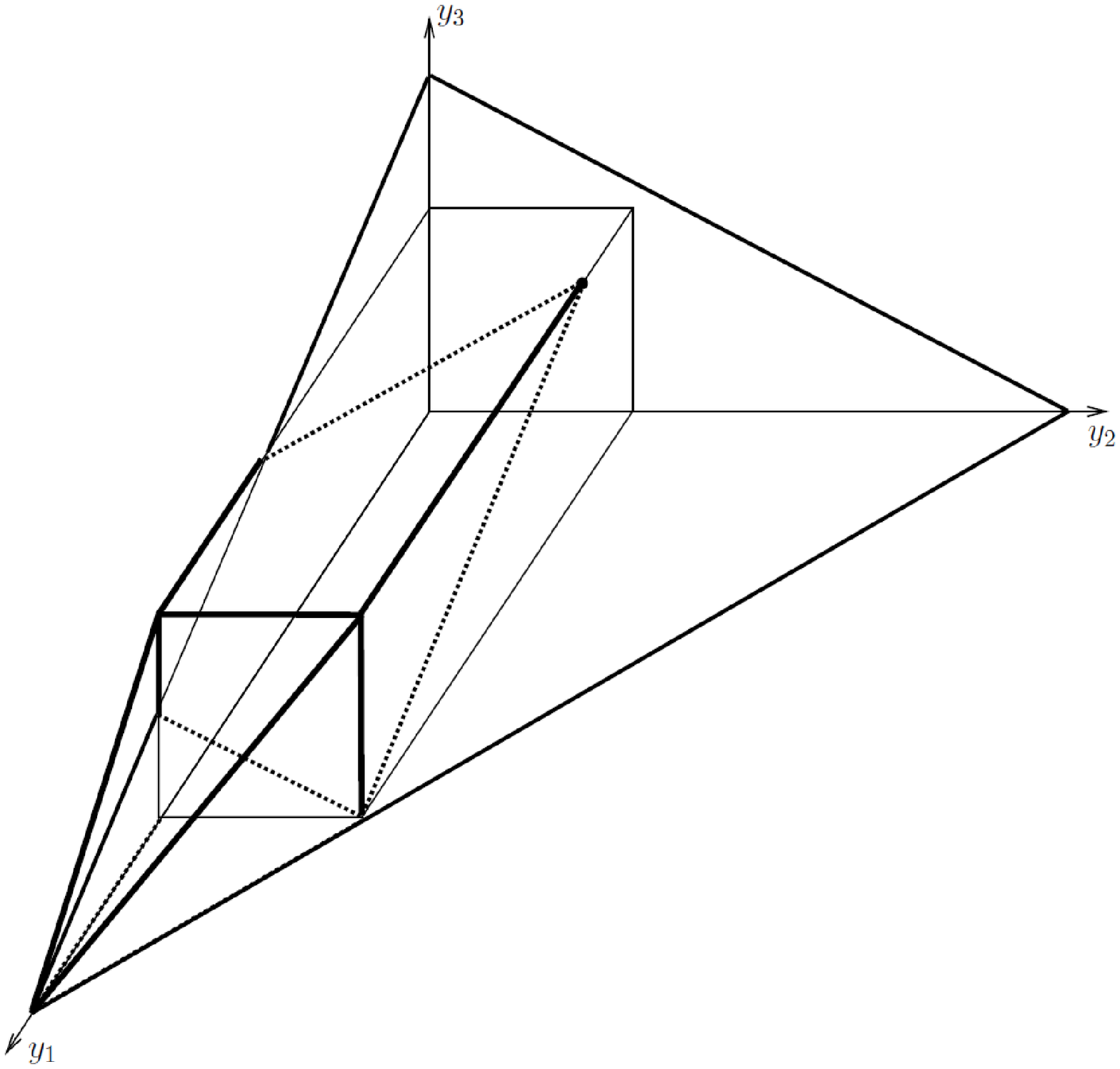}\\
\end{center}

\begin{lemma}
Under the previous conditions, we have
$$
N(S,\alpha) \bigcap \left\{ y_1 \geq \frac{\alpha}{a_1} - a_2 \right\} \subset D.
$$
\end{lemma}

\begin{proof}
Let us take an integral point $P = (y_1,...,y_k) \in N(S,\alpha)$, with
$$
\frac{\alpha}{a_1} - a_2 \leq y_1 \leq \frac{\alpha}{a_1},
$$
and let us write 
$$
y_1 = \frac{\alpha}{a_1} - \lambda_1 a_2 \; \Longrightarrow \lambda_1 = \frac{\alpha/a_1 - y_1}{a_2},
$$
and clearly $0 \leq \lambda_1 \leq 1$. Obviously, we have to define
$$
\lambda_i = \frac{y_i}{\lambda_1 a_1}, \mbox{ for } i=2,...,k;
$$
in order to write $P$ as in the definition of $D$. 

It is straightforward that $\lambda_i \geq 0$. On the other hand, one has that, $P$ being in $N(S,\alpha)$,
$$
\alpha \geq a_1y_1+...+a_ky_k = \alpha - \lambda_1 a_1a_2 + \sum_{i=2}^k a_iy_i
$$
and then, for $i=2,...,k$;
$$
a_iy_i \leq a_2y_2 + ... + a_ky_k \leq \lambda_1a_1a_2 \leq \lambda_1a_1a_i,
$$
which implies $y_i \leq \lambda_1a_1$ and therefore $\lambda_i \leq 1$, for $i=2,...,k$.
\end{proof}

We have finally proved:

\begin{proposition}
With the previous definitions and assumptions, we have 
$$
N(S,\alpha) \subset C \cup D.
$$
\end{proposition}

\begin{corollary}
With the previous definitions and assumptions, we have 
$$
n(S,\alpha) \leq \sharp \left( C \cup D \cap \Z^k_{\geq 0} \right).
$$
\end{corollary}

The number of integral points in $C$ is easy to compute:
$$
\sharp \left( C \cap \Z^k_{\geq 0} \right) = a_1^{k-1} \left( \left\lfloor \frac{\alpha}{a_1} - a_2 \right\rfloor + 1 \right)
$$

If $a_1$ does not divide $\alpha$, we can alternatively express it as  
$$
\sharp \left( C \cap \Z^k_{\geq 0} \right) = a_1^{k-1} \left( \left\lceil \frac{\alpha}{a_1} \right\rceil - a_2 \right).
$$

In order to find the number of integral points in $D$, let us fix our attention in a $y_1$--constant level of the pyramid. That is, fix $\lambda_1$ such that 
$$
\frac{\alpha}{a_1} - \lambda_1 a_2 \in \Z,
$$
and then the set
$$
D \bigcap \left\{ y_1 = \frac{\alpha}{a_1} - \lambda_1 a_2 \right\} \bigcap \Z_{\geq 0}^k
$$
is once again a $(k-1)$--hypercube determined by the vertices
$$
\lambda_1 (0,...,0,\stackrel{(i)}{a_1},0,...,0) \mbox{ for } i=2,...,k;
$$
which have therefore $\left( \left\lfloor \lambda_1 a_1 \right\rfloor +1 \right)^{k-1}$ integral points. 

All we need therefore is a precise description of the $\lambda_1$ which verify 
$$
\frac{\alpha}{a_1} - \lambda_1 a_2 \in \Z.
$$

There must then be a $\lambda \in \Z$ such that
$$
\frac{\alpha}{a_1} - \lambda_1 a_2 = \left\lfloor \frac{\alpha}{a_1} \right\rfloor - \lambda,
$$
and this $\lambda$ must verify $0 \leq \lambda \leq a_2-1$, for 
$$
\alpha/a_1-a_2 < y_1 \leq \alpha/a_1
$$ 
to hold. As
$$
\lambda_1 = \frac{\lambda + \alpha/a_1 - \lfloor \alpha/a_1 \rfloor}{a_2} =\frac{\lambda + \{\alpha/a_1\}}{a_2},
$$
we have the number of points at the level determined by $\lambda$ is
$$
\sharp \left( D \bigcap \left\{ y_1 = \left\lfloor \frac{\alpha}{a_1} \right\rfloor - \lambda \right\} \bigcap \Z_{\geq 0}^k \right) = \left( \left\lfloor a_1 \cdot \frac{\lambda + \{ \alpha/a_1 \}}{a_2} \right\rfloor +1 \right)^{k-1}
$$
and 
$$
\sharp \left( D \bigcap \Z_{\geq 0}^k \right) = \sum_{\lambda=0}^{a_2-1} \left( \left\lfloor a_1 \cdot \frac{\lambda + \{ \alpha/a_1 \} }{a_2} \right\rfloor +1 \right)^{k-1}
$$

\begin{theorem}
Let $S = \langle a_1,...,a_k \rangle$ be a numerical semigroup, $\alpha \geq a_1a_2$ an integer. Then 
$$
n(S,\alpha) \leq  a_1^{k-1} \left( \left\lfloor \frac{\alpha}{a_1} - a_2 \right\rfloor + 1 \right) + \sum_{\lambda=0}^{a_2-1} \left( \left\lfloor a_1 \cdot \frac{\lambda + \{ \alpha/a_1 \}}{a_2} \right\rfloor +1 \right)^{k-1}.
$$
\end{theorem}

\begin{corollary}
Let $S = \langle a_1,...,a_k \rangle$ be a numerical semigroup, $\alpha \geq a_1a_2$ an integer. Then 
$$
n(S,\alpha) \leq a_1^{k-1} \left\lfloor \frac{\alpha}{a_1} \right\rfloor
$$
\end{corollary}

\begin{proof}
Directly, extend the prism $C$ up to $y_1=\alpha/a_1$. Indirectly, as $0 \leq \lambda \leq a_2-1$ we have
$$
\frac{\lambda + \{ \alpha/a_1 \}}{a_2} < 1
$$
and therefore
$$
\left\lfloor a_1 \cdot \frac{\lambda + \{ \alpha/a_1 \}}{a_2} \right\rfloor +1 \leq a_1
$$
hence
$$
\sharp \left( D \bigcap \Z_{\geq 0}^k \right) \leq \sum_{\lambda=0}^{a_2-1} a^{k-1}
$$
and finally this implies
$$
n(S,\alpha) \leq  a_1^{k-1} \left( \left\lfloor \frac{\alpha}{a_1} - a_2 \right\rfloor + 1 \right) + (a_2-1)a_1^{k-1} = a_1^{k-1} \left\lfloor \frac{\alpha}{a_1} \right\rfloor,
$$
as stated.
\end{proof}

We have been working under the assumption $\alpha \geq a_1a_2$. The other case $\alpha \leq a_1a_2$ or, otherwise said
$$
\frac{\alpha}{a_1} - a_2 \leq 0,
$$
correspond to the following geometric situation: when we construct the prism, the $(k-1)$--hypercube in the basis is already out of $n(S,\alpha)$. We can still consider a pyramid $D$, much in the same fashion as above, although we must not be very optimistic with respect to the accuracy of the bound.

In this case, it is enough to consider the $(k-1)$--hypercube on $y_1=0$ to have side length $\alpha/a_2$. 

\begin{center}
\includegraphics[scale=0.3]{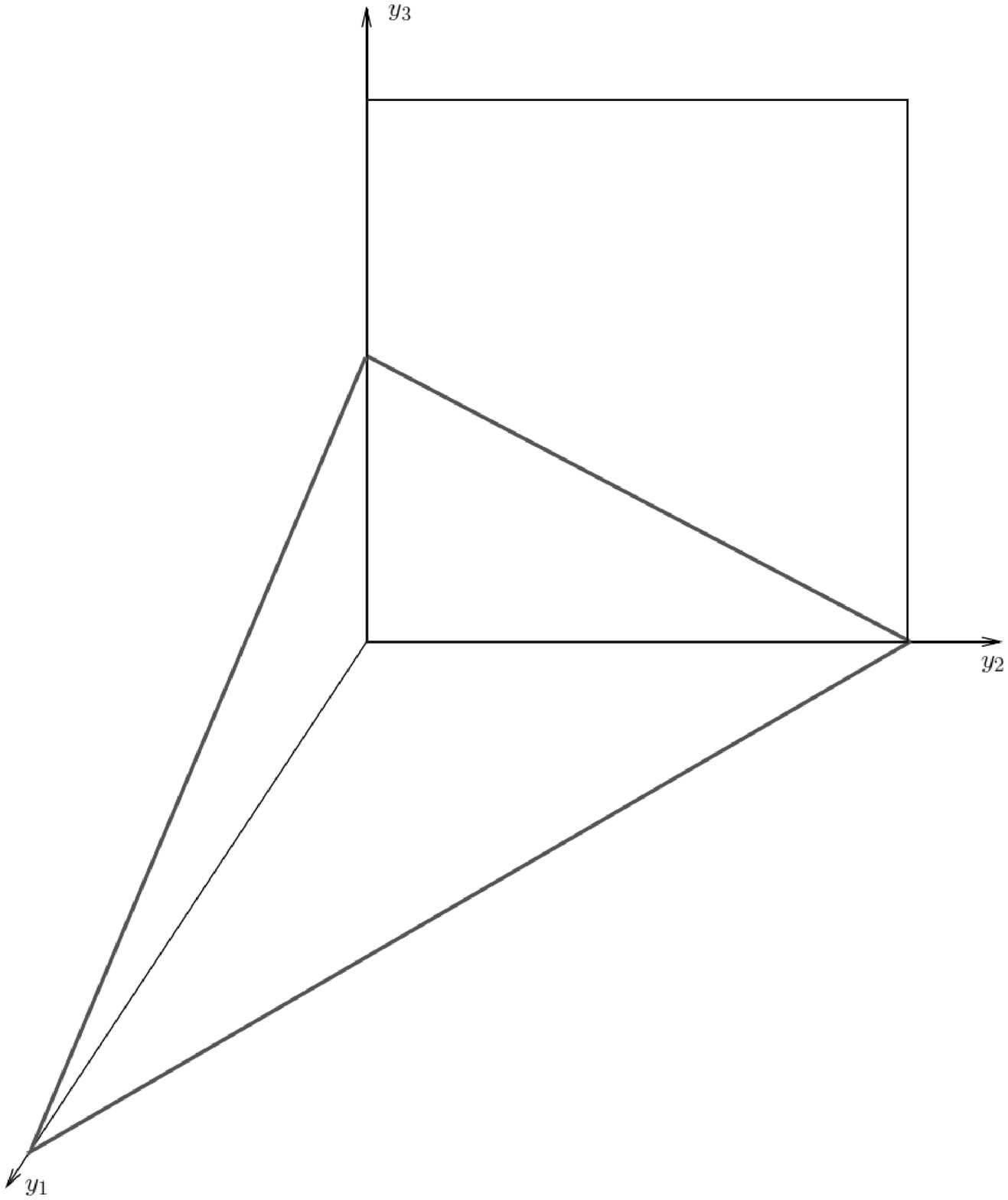}\\
\end{center}

We will not fill the technical details for this case, which are pretty similiar to the previous one. Let us mention that now the pyramid is:
$$
V = \left( \frac{\alpha}{a_1},0,...,0 \right),
$$
$$
D = \left\{ V + \lambda_1 \left( -\frac{\alpha}{a_1},0,...,0 \right) + \sum_{i=2}^k \lambda_1\lambda_i \left(0,...,0,\stackrel{(i)}{\frac{\alpha}{a_2}},0,...,0 \right) \; | \; 0 \leq \lambda_i \leq 1, \; \forall i \right\}.
$$

\begin{center}
\includegraphics[scale=0.3]{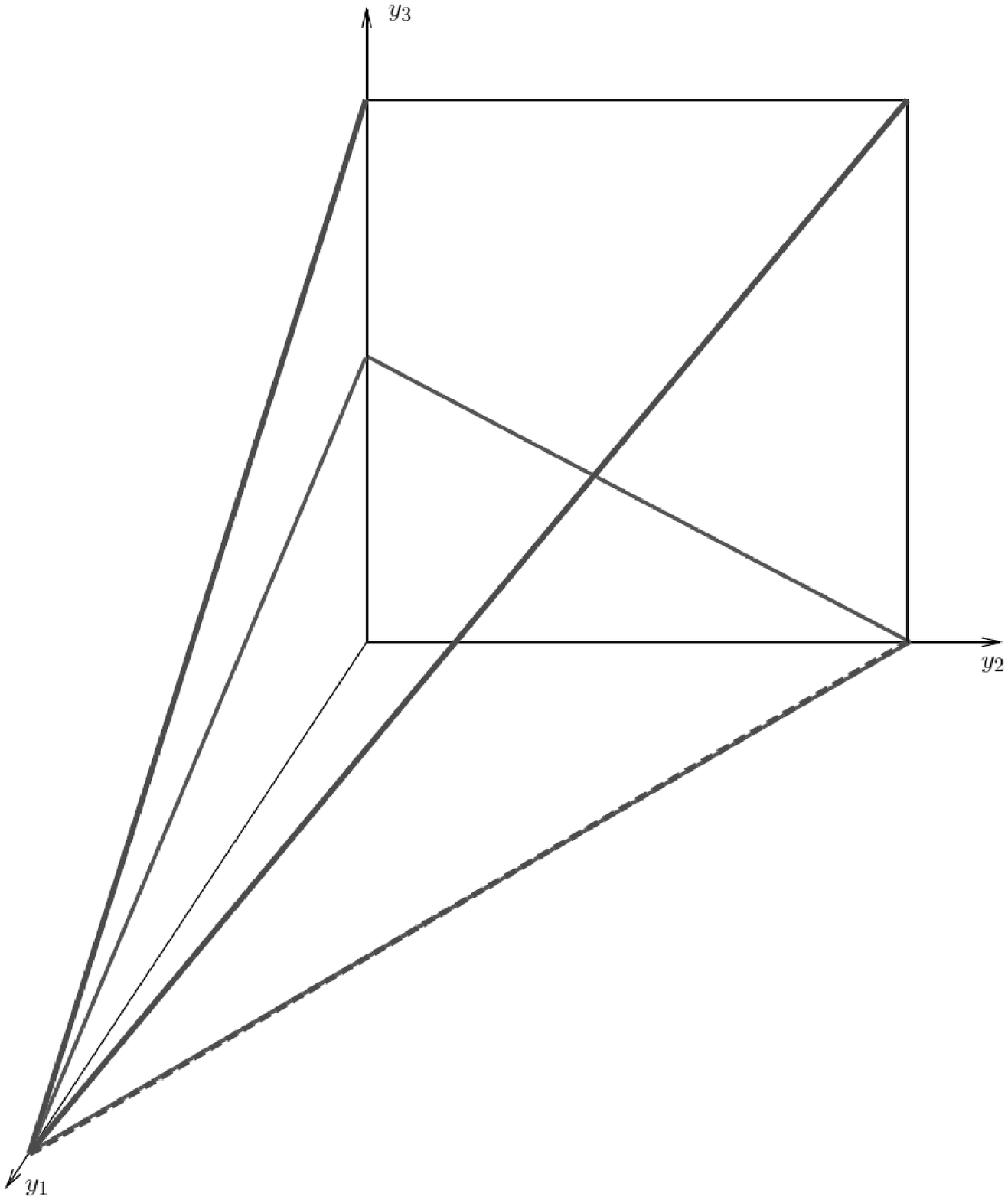}\\
\end{center}

In this case, we can simply consider a certain $\lambda \in \Z$ such that
$$
0 \leq \lambda \leq \left\lfloor \frac{\alpha}{a_1} \right\rfloor,
$$
which determines as above a $y_1$--constant level which is again a $(k-1)$--hypercube, defined in this case by the points
$$
\left( \lambda,...,0,\stackrel{(i)}{\frac{\alpha-\lambda a_1}{a_2}},0,...,0 \right), \mbox{ for } i=2,...,n.
$$

The equivalent result comes from adding up integral points in each $y_1$-- constant level and is therefore as follows:

\begin{theorem}
Let $S = \langle a_1,...,a_k \rangle$ be a numerical semigroup, $0 \leq \alpha \leq a_1a_2$ an integer. Then 
$$
n(S,\alpha) \leq  \sum_{\lambda=0}^{\left\lfloor \alpha/a_1 \right\rfloor} \left( \left\lfloor \frac{\alpha - \lambda a_1}{a_2} \right\rfloor +1 \right)^{k-1}.
$$
\end{theorem}

\begin{corollary}
In the above conditions,
$$
n(S) \leq \sum_{\lambda=0}^{a_2} \left( \left\lfloor a_1 \frac{a_2-\lambda}{a_2} \right\rfloor + 1 \right)^{k-1} + f(S) - a_1a_2.
$$
\end{corollary}

\begin{proof}
As $a_1a_2 \geq f(S)$, we can take $\alpha = a_1a_2$ and we have that
$$
n(S,a_1 a_2) = a_1a_2 - f(S) + n(S).
$$
\end{proof}

Much work is yet to be done. Most probably a better version of the GLY Conjecture will lead to a more precise results and there might be wiser ways to bound $n(S)$ than the {\em ''prism + pyramid''} method developed here. 

We hope this work sheds some light to the power and usefulness of Groebner bases in the study of numerical semigroups.

\section{Acknowledgments}

Thanks are due to Jorge Ram\'{\i}rez--Alfons\'{\i}n, who hosted the first author during her stay at Montpellier and has been tirelessly helpful.

The authors also thank Pedro Garc\'{\i}a--S\'anchez for his advice and for pointing out the reference \cite{Herzog} to them.

\end{document}